\documentclass{amsart}
\usepackage[utf8x]{inputenc}

\usepackage{amsthm,amsmath,amssymb,amstext,amsfonts,color}

\usepackage{mathtools}

\DeclarePairedDelimiter\floor{\lfloor}{\rfloor}

\usepackage{enumitem}
\newcommand{\subscript}[2]{$#1 _ #2$}

\usepackage{graphicx,float}

\newcommand{\field}[1]{\mathbb{#1}}

\newcommand{\N}{\field{N}}

\numberwithin{equation}{section}
\newtheorem{theorem}{Theorem}[section]
\newtheorem{lemma}[theorem]{Lemma}
\newtheorem{corollary}[theorem]{Corollary}

\theoremstyle{remark}
\newtheorem*{remark}{Remark}

\renewenvironment{proof}[1][Proof]{\begin{trivlist}
\item[\hskip \labelsep {\bfseries #1:}]}{\qed\end{trivlist}}

\title{On partition identities of Capparelli and Primc}

\author{Jehanne Dousse}
\address{Univ Lyon, CNRS, Université Claude Bernard Lyon 1, UMR5208, Institut Camille Jordan, F-69622 Villeurbanne, France}
\email{dousse@math.cnrs.fr}

\begin{document}

\begin{abstract}
We show that, up to multiplication by a factor $\frac{1}{(cq;q)_{\infty}}$, the weighted words version of Capparelli's identity is a particular case of the weighted words version of Primc's identity. We prove this first using recurrences, and then bijectively. We also give finite versions of both identities.
\end{abstract}

\maketitle

\section{Introduction and statement of results}
\subsection{Historical background}
A \textit{partition} $\lambda$ of a positive integer $n$ is a non-increasing sequence of natural numbers whose sum is $n$, the partitions of $4$ being $4, 3+1, 2+2, 2+1+1,$ and $ 1+1+1+1.$ The number $n$ is called the \emph{weight} of $\lambda$. Let us recall, for $n \in \mathbb{N} \cup \{\infty\}$, the classical $q$-series notation
$$
(a;q)_n := \prod_{k=0}^{n-1} (1-aq^k).
$$

Connections between partition identities and representation theory have been a major subject of interest over the last few decades, beginning with Lepowsky and Wilson's representation theoretic proof of the famous Rogers-Ramanujan identities \cite{Lepowsky,Lepowsky2}.

\begin{theorem}[The Rogers-Ramanujan identities]
\label{th:RR}
Let $i=0$ or $1$. Then
\begin{equation} \label{R-R}
\sum_{n \geq 0} \frac{q^{n^2+ (1-i)n}}{(q;q)_n} = \frac{1}{(q^{2-i};q^5)_{\infty}(q^{3+i};q^5)_{\infty}}.
\end{equation} 
\end{theorem}

Lepowsky and Wilson showed that, after multiplying both sides of \eqref{R-R} by $(-q;q)_{\infty}$, the right-hand side is the principally specialised Weyl-Kac character formula for level $3$ standard modules of $A_1^{(1)}$ \cite{Le-Mi1,Le-Mi2}, and the left-hand side corresponds to bases constructed from vertex operators. 

The Rogers-Ramanujan identities can also be seen as combinatorial identities on partitions.

\begin{theorem}[Rogers-Ramanujan, combinatorial version]
\label{th:RR comb}
Let $i=0$ or $1$. For all non-negative integers $n$, the number of partitions of $n$ into parts differing by at least $2$ and having at most $i$ ones is equal to the number of partitions of $n$ into parts congruent to $\pm (2-i)$ modulo $5$.
\end{theorem}

The approach of Lepowsky and Wilson was then extended and modified by several authors to treat other levels and other Lie algebras, leading to many interesting new Rogers-Ramanujan type identities which were previously unknown to combinatorialists. For some examples, see \cite{Capparelli,Capparelli2,Kanaderussellstaircase,Meurman,Meurman2,Meurman3,Primc1,PrimcSikic,Siladic} and the references therein.
On the other hand, combinatorialists have been working on combinatorial proofs and refinements of these new identities \cite{AllAndGor,Andrewscap,Doussesil,Doussesil2,DoussePrimc,DousseCapa}.

\medskip
The purpose of this paper is to establish a connection between two partition identities from representation theory: Capparelli's identity and Primc's identity. Let us now present these two theorems in detail.

\subsection{Capparelli's identity}
A good example of the interplay between combinatorics and representation theory is Capparelli's identity, which was conjectured by Capparelli in \cite{Capparelli} by studying the Lie algebra $A_2^{(2)}$ at level $3$.

\begin{theorem}[Capparelli]
\label{th:capa}
Let $C(n)$ denote the number of partitions of $n$ into parts $>1$ such that parts differ by at least $2$, and at least $4$ unless consecutive parts add up to a multiple of $3$.  Let $D(n)$ denote the number of partitions of $n$ into distinct parts not congruent to $\pm 1 \pmod{6}$. Then for every positive integer $n$, $C(n) = D(n)$.   
\end{theorem}

The first proof was provided by Andrews in~\cite{Andrewscap} and used $q$-trinomial coefficients and recurrences. The identity was then proved bijectively, refined and generalised by Alladi, Andrews and Gordon in~\cite{AllAndGor}, where they used the method of weighted words. Soon after, it was also proved via representation theoretic techniques by Capparelli in~\cite{Capparelli2} and by Tamba-Xie in~\cite{Xie}. In \cite{Meurman2}, Meurman and Primc later showed that one can recover Capparelli's identity by studying the $(1,2)$-specialisation of the character formula for level $1$ modules in $A_1^{(1)}$. Finally, more combinatorial work has been done concerning this identity in \cite{Br-Ma1,Be-Un1,DousseCapa,Fu-Ze1,Kanaderussellstaircase,Si1}.

\bigskip

In this paper we focus on the weighted words approach of Alladi, Andrews and Gordon in~\cite{AllAndGor}. The principle is to prove a ``non-dilated'' version of Capparelli's identity on coloured partitions, which recovers the original identity under certain transformations called dilations. In addition to providing a refinement of Capparelli's identity, the advantage of this method is that one can perform other dilations and obtain infinitely many new combinatorial identities.

Let us now explain their method. Note that we put tildes on their original colour names, as we want to avoid any confusion with the new colours we will introduce in this paper to establish the connection with Primc's identity.

They considered partitions into natural numbers in three colours, $\tilde{a}$, $\tilde{b}$, and $\tilde{c}$ , with no part $1_{\tilde{a}}$ or $1_{\tilde{b}}$, 
with the ordering
\begin{equation}
\label{order'}
2_{\tilde{b}} < 1_{\tilde{c}} < 2_{\tilde{a}} < 3_{\tilde{b}} < 2_{\tilde{c}} < 3_{\tilde{a}} < \cdots ,
\end{equation}  
satisfying the difference conditions in the matrix
\begin{equation} \label{Cappdiffmatrix}
\tilde{C}=\bordermatrix{\text{} & \tilde{a} & \tilde{b} & \tilde{c} \cr \tilde{a} & 2 & 0 & 2 \cr \tilde{b} & 2 & 2 & 3 \cr \tilde{c} & 1 & 0 & 1}.
\end{equation}
Here the entry $(x,y)$ in the matrix $\tilde{C}$ gives the minimal difference between successive parts where the larger one is coloured $x$ and the smaller one is coloured $y$.   

They proved the following non-dilated version of Capparelli's identity.
\begin{theorem}[Alladi-Andrews-Gordon, non-dilated version of Capparelli's identity ]
Let $\tilde{C}(n;i,j)$ denote the number of partitions of $n$ into coloured integers satisfying the difference conditions in the matrix $\tilde{C}$, with no part $1_{\tilde{a}}$ or $1_{\tilde{b}}$, having $i$ parts coloured $\tilde{a}$ and $j$ parts coloured $\tilde{b}$.
\label{th:capa-aag}
We have
\begin{equation} \label{eq:capa-aag}
\sum_{n,i,j \geq 0} \tilde{C}(n;i,j) \tilde{a}^i \tilde{b}^j q^n =  (-q;q)_{\infty}(-\tilde{a}q^2;q^2)_{\infty}(-\tilde{b}q^2;q^2)_{\infty}.
\end{equation}
\end{theorem}
Note that to obtain an infinite product, one cannot keep track of the number of parts coloured $\tilde{c}$ in the generating function.

Under the dilations 
$$q \rightarrow q^3, \quad \tilde{a} \rightarrow \tilde{a}q^{-2}, \quad \tilde{b} \rightarrow \tilde{b}q^{-4},$$
which correspond to the following transformations of the coloured integers
$$ k_{\tilde{a}} \rightarrow (3k-2)_{\tilde{a}}, \quad  k_{\tilde{b}} \rightarrow (3k-4)_{\tilde{b}}, \quad  k_{\tilde{c}} \rightarrow (3k)_{\tilde{c}}, \quad$$
the order \eqref{order'} becomes the natural ordering
$$2_{\tilde{b}} < 3_{\tilde{c}} < 4_{\tilde{a}} < 5_{\tilde{b}} < 6_{\tilde{c}} < 7_{\tilde{a}} < \cdots,$$
and the difference conditions in the matrix $\tilde{C}$ of \eqref{Cappdiffmatrix} become
\begin{equation} 
\bordermatrix{\text{} & \tilde{a} & \tilde{b} & \tilde{c} \cr \tilde{a} & 6 & 2 & 4 \cr \tilde{b} & 4 & 6 & 5 \cr \tilde{c} & 5 & 4 & 3},
\end{equation}
which is simply another formulation of the difference conditions defining the partitions counted by $C(n)$ in Theorem \ref{th:capa}. Under the same dilations, the infinite product in \eqref{eq:capa-aag} becomes the generating function for the partitions counted by $D(n)$. With the two extra parameters $a$ and $b$, this gives the following refinement of Capparelli's identity.

\begin{corollary}[Alladi-Andrews-Gordon] \label{cor:capa-aag}
Let $C(n;i,j)$ and $D(n;i,j)$ denote the number of partitions counted by $C(n)$ and $D(n)$, respectively, in Theorem \ref{th:capa}, having $i$ parts congruent to $1$ modulo $3$ and $j$ parts congruent to $2$ modulo $3$. Then for all $n, i , j \in \N,$ $C(n;i,j) = D(n;i,j)$.
\end{corollary}

\subsection{Primc's identity}
We now describe Primc's identity and its weighted word version.

In \cite{Primc2}, Primc established a connection between the difference conditions in certain vertex operator constructions and energy functions of certain perfect crystals. He further developed his ideas in \cite{Primc} to prove new partition identities arising from crystal base theory. His approach relies not only on the Weyl-Kac character formula as was done by Lepowsky and Wilson, but also on the crystal base character formula of Kang, Kashiwara, Misra, Miwa, Nakashima and Nakayashiki \cite{KMN2}.

Here, we focus on one of the identities of \cite{Primc}.
Consider partitions into natural numbers in four colours $a,b,c,d$, with the ordering
\begin{equation}
\label{orderprimc}
1_{a} < 1_{b} < 1_{c} <1_{d} <2_{a} < 2_{b} <2_{c} < 2_{d} < \cdots ,
\end{equation}
satisfying the difference conditions in the matrix
\begin{equation} \label{Primcmatrix}
P=\bordermatrix{\text{} & a & b & c & d \cr a & 2&1&2&2 \cr b &1&0&1&1 \cr c &0&1&0&2 \cr d&0&1&0&2}.
\end{equation}

Primc conjectured the following.
\begin{theorem}[Primc]
\label{th:primc}
We have
$$
\sum_{\lambda} q^{\sum_{k \geq 1}\big( (2k-1)A_k(\lambda) + 2k (B_k(\lambda) + C_k(\lambda)) + (2k+1)D_k(\lambda)\big)} = \frac{1}{(q;q)_{\infty}}, 
$$
where the sum is over the coloured partitions $\lambda$ satisfying the difference conditions given by \eqref{Primcmatrix} and where $A_k(\lambda)$ (resp. $B_k(\lambda), C_k(\lambda), D_k(\lambda)$) denotes the number of parts $k$ of colour $a$ (resp. $b, c ,d$) in $\lambda$. 
\end{theorem}
In other words, if the coloured integers in \eqref{orderprimc} are transformed by
\begin{equation} \label{Primcdilation}
k_{a} \rightarrow (2k-1)_a, \quad k_{b} \rightarrow (2k)_b, \quad k_{c} \rightarrow (2k)_c, \quad k_{d} \rightarrow (2k+1)_d,
\end{equation}
the generating function for the resulting coloured partitions with the difference conditions given by
$$
P_{dil2}=\bordermatrix{\text{} & a & b & c & d \cr a & 4&1&3&2 \cr b &3&0&2&1 \cr c &1&2&0&3 \cr d&2&3&1&4},
$$
where one does not keep track of the number of parts of each colour, is equal to the generating function for ordinary partitions.

\medskip

In \cite{DoussePrimc}, Lovejoy and the author proved a non-dilated version of Primc's theorem, using the author's new variant of the method of weighted words \cite{Doussesil2,Dousseunif}.

\begin{theorem}[Dousse-Lovejoy, non-dilated version of Primc's identity]
\label{th:primcnondil}
\ \\
Let $P(n;k,\ell,m)$ denote the number of four-coloured partitions of $n$ with the ordering \eqref{orderprimc} and matrix of difference conditions \eqref{Primcmatrix}, having $k$ parts coloured $a$, $\ell$ parts coloured $c$, and $m$ parts coloured $d$. Then
$$\sum_{n,k,\ell,m \geq 0} P(n;k,\ell,m) q^n a^k c^{\ell} d^m = \frac{(-aq;q^2)_{\infty}(-dq;q^2)_{\infty}}{(q;q)_{\infty}(cq;q^2)_{\infty}}.$$
\end{theorem}

Performing the dilations
\begin{equation} \label{Primcrefineddilation}
q \rightarrow q^2, \quad a \rightarrow aq^{-1}, \quad c \rightarrow c, \quad d\rightarrow dq,
\end{equation}
one obtains a refinement of Primc's Theorem \ref{th:primc}.

\subsection{Statement of results}
The goal of this paper is to establish a correspondence between Capparelli's and Primc's identities. To do so, we will only consider non-dilated versions of these theorems, such as Theorem \ref{th:capa-aag} and Theorem \ref{th:primcnondil}. Indeed, we have just seen that this is more general and that the original theorems can be recovered under particular dilations. From now on, we shall not explicitly write ``non-dilated version" anymore, as all theorems considered are at the non-dilated level.

When looking at the difference condition matrices for Capparelli's identity \eqref{Cappdiffmatrix} and Primc's identity \eqref{Primcmatrix}, one can argue that they do not seem very similar: they do not have the same size, and their entries do not seem to exhibit any similar pattern. From a representation theoretic point of view, the two identities do not seem to be related a priori either: Primc's identity comes from the study of crystal bases of $A_1^{(1)}$, while Capparelli's identity does not seem related to crystal bases and originated from a vertex operator construction on the level $3$ modules of $A_2^{(2)}$ (though Meurman and Primc established a connection with $A_1^{(1)}$ in \cite{Meurman2}).

But when looking at the infinite products, we observe some similarities: the product in Theorem \ref{th:capa-aag} is 
$$(-q;q)_{\infty}(-\tilde{a}q^2;q^2)_{\infty}(-\tilde{b}q^2;q^2)_{\infty},$$
while the one in Theorem \ref{th:primcnondil} is
$$\frac{(-aq;q^2)_{\infty}(-dq;q^2)_{\infty}}{(q;q)_{\infty}(cq;q^2)_{\infty}}.$$
If we set $c=1$, $a=\tilde{a}q$, $d=\tilde{b}q$ in the second one, we obtain exactly
the first one multiplied by the factor $\frac{1}{(q;q)_{\infty}}$, which hints at a potential connection between the two theorems.

However, a simple connection between the infinite products does not necessarily mean that there is a simple connection between the difference conditions of both theorems. For example, the non-dilated version of Schur's theorem \cite{Alladi} involves simple difference conditions on integers in $3$ colours, while the non-dilated version \cite{Doussesil2} of Siladi\'c's theorem \cite{Siladic} -- another partition identity from representation theory -- involves quite intricate difference conditions on integers in $8$ colours. Nonetheless, they both have the same infinite product generating function $(-aq;q)_{\infty}(-bq;q)_{\infty}$. Very recently, Konan gave a bijective proof of the non-dilated version of Siladi\'c's theorem \cite{Konan}, shedding light on the connection between the two identities.

\medskip
In this paper, we show that, up to the $\frac{1}{(q;q)_{\infty}}$ above-mentioned factor, Capparelli's identity is actually the particular case $b=c$ of Primc's identity. But before stating our main theorem rigorously, let us rewrite Capparelli's identity in a form that makes the connection with Primc's identity more explicit.

Start from Theorem \ref{th:capa-aag} and do the following transformations on coloured integers in \eqref{order'}:
$$k_{\tilde{a}} \rightarrow (k-1)_d, \quad k_{\tilde{b}} \rightarrow (k-1)_a, \quad
k_{\tilde{c}} \rightarrow k_c,$$
which corresponds to the following transformations in the generating functions:
$$ \tilde{a}=dq^{-1}, \quad \tilde{b} = aq^{-1}, \quad \tilde{c}=c.$$

We obtain partitions into natural numbers in three colours, $a$, $c$, and $d$, the ordering \eqref{order'} becomes
\begin{equation}
\label{order}
1_a < 1_c < 1_d < 2_a < 2_c < 2_d < \cdots ,
\end{equation}  
and the difference condition matrix \eqref{Cappdiffmatrix} becomes
\begin{equation} \label{Capmatrix}
C=\bordermatrix{\text{} & a & c & d \cr a & 2 & 2 & 2 \cr c & 1 & 1 & 2 \cr d & 0 & 1 & 2}.
\end{equation} 

Thus Theorem \ref{th:capa-aag} can be reformulated as follows.
\begin{theorem}[Reformulation of Capparelli's identity]
Let $C(n;i,j)$ denote the number of partitions of $n$ into coloured integers satisfying the difference conditions in matrix $C$, having $i$ parts coloured $a$ and $j$ parts coloured $d$.
\label{th:capa-dou}
We have
\begin{equation*}
\sum_{n,i,j \geq 0} C(n;i,j)a^i d^j q^n =  (-q;q)_{\infty}(-aq;q^2)_{\infty}(-dq;q^2)_{\infty}.
\end{equation*}
\end{theorem}
To recover Theorem \ref{th:capa} and Corollary \ref{cor:capa-aag} from Theorem \ref{th:capa-dou}, one should now perform the dilations 
$$q \rightarrow q^3, \quad a \rightarrow aq^{-1}, \quad d \rightarrow dq.$$

The matrix $C$ from Theorem \ref{th:capa-dou} also appeared in Primc's paper \cite{Primc} as the energy matrix of an ``almost perfect'' $A_1^{(1)}$--crystal (see Section \ref{sec:crystal} for details), together with identities for the particular dilations
$ q \rightarrow q^2, \ a \rightarrow q^{-1}, \ d \rightarrow q$ and $ q \rightarrow q^3, \ a \rightarrow q^{-2}, \ d \rightarrow q^2$. 
However, Primc said that these identities are \emph{`not related to the crystal base theory, at least not in any obvious way.'} Indeed, they were proved by him and Meurman in \cite{Meurman2} by using the vertex operator algebra construction for level one $A_1^{(1)}$-modules.

\medskip
In this paper we show that Theorem \ref{th:capa-dou} (and thus also the identities mentioned by Primc) is actually connected to the crystal base theory, as it is a particular case of Theorem \ref{th:primcnondil} (the non-dilated version of Primc's identity coming from crystal base theory).

Let us state our main theorem. Define $G^C_k (q;a,c,d)$ to be the generating function for partitions into coloured integers \eqref{order} satisfying the difference conditions from Capparelli's identity \eqref{Capmatrix}, with the added condition that the largest part is at most $k$.
In the same way, define $G^P_k (q;a,b,c,d)$ to be the generating function for partitions into coloured integers \eqref{orderprimc} satisfying the difference conditions from Primc's identity \eqref{Primcmatrix}, with the added condition that the largest part is at most $k$. In these generating functions, the power of $a$ (resp. $b, c, d$) counts the number of parts coloured $a$ (resp. $b, c, d$) in the partition.

\begin{theorem}
\label{th:main}
For all positive integers $k$, we have
$$\frac{G^C_k (q;a,c,d)}{(cq;q)_k}= G^P_k (q;a,c,c,d).$$
\end{theorem}

\begin{remark}
In Theorem \ref{th:capa-dou}, one needs to set the variable $c$ to be equal to $1$ (i.e. not keep track of the number of parts coloured $c$) to obtain an infinite product generating function. Similarly, in Theorem \ref{th:primcnondil}, one needs to set the variable $b$ to be equal to $1$ (i.e. not keep track of the number of parts coloured $b$).
However we see here that the generating functions $G^C_k (q;a,c,d)$ and $G^P_k (q;a,c,c,d)$ are equal even when keeping track of all the statistics $a, c, d$ from Capparelli's identity.
\end{remark}

In terms of partitions, Theorem \ref{th:main} can be expressed in the following way. Let us define $\mathcal{C}$ (resp. $\mathcal{P}$) to be the set of coloured partitions satisfying the order \eqref{order} (resp. \eqref{orderprimc}) and difference conditions \eqref{Capmatrix} (resp. \eqref{Primcmatrix}).

\begin{theorem}[Combinatorial version]
\label{th:comb}
Let $\mathcal{C}(n;k;i,j,\ell)$ denote the number of partition pairs $(\lambda,\mu)$ of total weight $n$, where $\lambda \in \mathcal{C}$ and $\mu$ is an unrestricted partition coloured $c$, having in total $i$ parts coloured $a$, $j$ parts coloured $c$, $\ell$ parts coloured $d$, and largest part at most $k$.
Let $\mathcal{P}(n;k;i,j,\ell)$ denote the number of partitions $\lambda \in \mathcal{P}$ of weight $n$, having $i$ parts coloured $a$, $j$ parts coloured $b$ or $c$, $\ell$ parts coloured $d$, and largest part at most $k$.
Then for all positive integers $n$ and $k$ and all non-negative integers  $i, j, \ell$,
$$\mathcal{C}(n;k;i,j,\ell)=\mathcal{P}(n;k;i,j,\ell).$$
\end{theorem}

We give a bijective proof of Theorem \ref{th:comb} in Section \ref{sec:bij}.

Thanks to Theorem \ref{th:main}, Capparelli's identity is now a corollary of Primc's identity. Indeed
\begin{align*}
\sum_{n,i,j \geq 0} C(n;i,j) a^i d^j q^n &=\lim_{k \rightarrow \infty} G^C_k (q;a,1,d)\\
&= \lim_{k \rightarrow \infty} (q;q)_{k} G^P_k (q;a,1,1,d)\\
&=(q;q)_{\infty} \sum_{n,i,j,\ell \geq 0} P(n;i,\ell,j) q^n a^i d^j\\
&= \frac{(-aq;q^2)_{\infty}(-dq;q^2)_{\infty}}{(q;q^2)_{\infty}}\\
&= (-q;q)_{\infty}(-aq;q^2)_{\infty}(-dq;q^2)_{\infty}.
\end{align*}

Even though the variable $c$ (resp. $b$) needs to be set equal to $1$ in Capparelli's (resp. Primc's) identity to obtain an infinite product, Theorem \ref{th:main} highlights the importance of these variables. Therefore it is interesting to find a formula for the generating functions of Capparelli's and Primc's identities with all colour variables.
Moreover, finding finite versions of partition identities has been a subject of interest in the recent years (see, e.g., \cite{Be-Un2} and \cite{GuoJouhetZeng}). We now present a finite version of both theorems with all colour variables.

\begin{theorem}[Finite version of Primc's identity]
\label{th:primc_fini}
We have, for every positive integer $k$,
$$G^P_k (q;a,b,c,d) = \left(1-bq^{k+1}\right) \sum_{j=0}^{k+1} \frac{u_{j}(a,b,c,d)q^{{k+1-j \choose 2}}}{(q;q)_{k+1-j}},$$
where for all $n \geq 0$,
$$
u_{2n}(a,b,c,d) = (1-b) \sum_{\ell=0}^n \frac{(-aq^{2\ell+1};q^2)_{n-\ell}(-dq^{2\ell+1};q^2)_{n-\ell}}{(bq^{2\ell};q^2)_{n-\ell+1}(cq^{2\ell+1};q^2)_{n-\ell}} \frac{q^{2\ell}}{(q;q)_{2\ell}},
$$
and
$$
u_{2n+1}(a,b,c,d) = (b-1) \sum_{\ell=0}^n \frac{(-aq^{2\ell+2};q^2)_{n-\ell}(-dq^{2\ell+2};q^2)_{n-\ell}}{(bq^{2\ell+1};q^2)_{n-\ell+1}(cq^{2\ell+2};q^2)_{n-\ell}} \frac{q^{2\ell+1}}{(q;q)_{2\ell+1}}.
$$
\end{theorem}

\begin{theorem}[Finite version of Capparelli's identity]
\label{th:capa_fini}
We have, for every positive integer $k$,
$$G^C_k (q;a,c,d) = (cq;q)_{k+1} \sum_{j=0}^{k+1} \frac{u_{j}(a,c,c,d)q^{{k+1-j \choose 2}}}{(q;q)_{k+1-j}},$$
where the sequence $(u_n(a,b,c,d))_{n \in \N}$ is defined as in Theorem \ref{th:primc_fini}.
\end{theorem}

Therefore, when $b=1$, Theorem \ref{th:primc_fini} becomes
\begin{corollary}
\label{cor:prim_fini_sans_b}
We have, for every positive integer $k$,
$$G^P_k (q;a,1,c,d) = \left(1-q^{k+1}\right) \sum_{j=0}^{\floor{\frac{k}{2}}} \frac{(-aq;q^2)_j(-dq;q^2)_j}{(q^2;q^2)_j(cq;q^2)_j}\frac{q^{{k+1-j \choose 2}}}{(q;q)_{k+1-j}},$$
\end{corollary}
and Theorem \ref{th:capa_fini} becomes
\begin{corollary}
\label{cor:capa_fini_sans_b}
We have, for every positive integer $k$,
$$G^C_k (q;a,1,d) = (q;q)_{k+1} \sum_{j=0}^{\floor{\frac{k}{2}}} \frac{(-aq;q^2)_j(-dq;q^2)_j}{(q^2;q^2)_j(cq;q^2)_j}\frac{q^{{k+1-j \choose 2}}}{(q;q)_{k+1-j}}.$$
\end{corollary}
It is then easy to recover the infinite product form by performing the change of variable $j= \floor{\frac{k}{2}}-j$, letting $k$ go to infinity, and using the fact that
$$\lim_{k \rightarrow \infty} \sum_{j=0}^{\floor{\frac{k}{2}}}\frac{q^{{2j \choose 2}}}{(q;q)_{2j}} = \lim_{k \rightarrow \infty} \sum_{j=0}^{\floor{\frac{k}{2}}}\frac{q^{{2j+1 \choose 2}}}{(q;q)_{2j+1}} = (-q;q)_{\infty}.$$

\medskip

The remainder of the paper is organised as follows.
In Section \ref{sec:rec}, we give an elementary proof of Theorem \ref{th:main} using recurrences. In Section \ref{sec:finite}, we prove the finite versions of Primc's and Capparelli's identities (Theorems \ref{th:primc_fini} and \ref{th:capa_fini}), again using recurrences. In Section \ref{sec:bij}, we prove Theorem \ref{th:comb} bijectively. Finally, in Section \ref{sec:crystal}, we conclude with some remarks related to crystals.

\section{Proof of Theorem~\ref{th:main} using recurrences}
\label{sec:rec}
In this section, we prove Theorem~\ref{th:main} by establishing recurrence equations using the difference condition matrices $C$ from \eqref{Capmatrix} and $P$ from \eqref{Primcmatrix}.

In addition to $G^C_k (q;a,c,d)$ and $G^P_k (q;a,b,c,d)$ defined in the introduction, we also define $E^C_k (q;a,c,d)$ (resp. $E^P_k (q;a,b,c,d)$) to be the generating function for partitions into coloured integers \eqref{order} (resp. \eqref{orderprimc}) satisfying the difference conditions from Capparelli's identity \eqref{Capmatrix} (resp. Primc's identity \eqref{Primcmatrix}), with the added condition that the largest part is equal to $k$. When there is no risk of confusion, we will omit the variables $a,b,c,d,q$ and write only $G^C_k, E^C_k,G^P_k, E^P_k.$

\medskip
We start with Capparelli's identity. By using the order \eqref{order} and the difference conditions from \eqref{Capmatrix}, we do a classical combinatorial reasoning on the largest part of the partition and obtain the following.
\begin{lemma}
\label{lem:CapAnd}
For all $k \geq 1,$ we have
\label{equations}
\begin{subequations}
\begin{align}
G^C_{k_{d}}-G^C_{k_{c}}=E^C_{k_{d}}&=dq^k \left(E^C_{k_{a}}+G^C_{(k-1)_{c}}\right), \label{eq1'} \\
G^C_{k_{c}}-G^C_{k_{a}}=E^C_{k_{c}}&=cq^k G^C_{(k-1)_{c}}, \label{eq2'} \\
G^C_{k_{a}}-G^C_{(k-1)_{d}}=E^C_{k_{a}}&=aq^k G^C_{(k-2)_{d}}.\label{eq3'}
\end{align}
\end{subequations}
\end{lemma}
\begin{proof}
We prove \eqref{eq1'}, and the other two equations work in the same way.
The equality $G^C_{k_{d}}-G^C_{k_{c}} = E^C_{k_{d}}$ follows directly from the definitions and the order \eqref{order}. Now, in a partition counted by $E^C_{k_{d}}$ we remove the largest part $k_d$, giving the factor $dq^k$. Using the difference conditions in \eqref{Capmatrix}, we see that the largest part in the resulting partition could be either $k_a$ or a part at most $(k-1)_c$. This gives the terms $E_{k_{a}}+G_{(k-1)_{c}}$.
\end{proof}

Moreover, to obtain the correct values for $G^C_k$ and $E^C_k$ for all coloured integers $k$, we need the initial conditions
\begin{align*}
E^C_{0_{a}}&=E^C_{0_{c}}=E^C_{0_{d}}=0,
\\G^C_{0_{a}}&=G^C_{0_{c}}=G^C_{0_{d}}=1,
\\G^C_{{-1}_d}&=1.
\end{align*}

Note that essentially the same reasoning as in Lemma \ref{lem:CapAnd} has already been made by Andrews in \cite{Andrewscap} and Alladi, Andrews, and Gordon in \cite{AllAndGor}, but at the dilated level.
Performing a chain of substitutions in the equations \eqref{eq1'}--\eqref{eq3'} as in \cite{Andrewscap}, we obtain the following recurrence equation for $k \geq 1$:
\begin{equation}
\label{eq:recCap}
\begin{aligned}
G^C_{k_d} &= \left(1+cq^{k}\right) G^C_{(k-1)_d} + \left(aq^k+dq^k+adq^{2k}\right) G^C_{(k-2)_d}
\\&+adq^{2k-1} \left(1-cq^{k-1}\right) G^C_{(k-3)_d}.
\end{aligned}
\end{equation}
Together with the initial conditions
\begin{equation}
\label{eq:cond_in_cap}
\begin{aligned}
G^C_{0_d} &= 1,
\\G^C_{-1_d} &= 1,
\\G^C_{-2_d} &= 0,
\end{aligned}
\end{equation}
the recurrence equation \eqref{eq:recCap} completely determines $G^C_k (q;a,c,d)= G^C_{k_d}$ for $k \geq 1$.

Let us now introduce the sequence $(H_k(q;a,b,c,d))$ defined by the following recurrence equation for $k \geq 0$ :
\begin{equation}
\label{eq:H}
\begin{aligned}
\left(1-cq^k\right)\left(1-bq^{k+1}\right)H_k(q;a,b,c,d) &= (1-bcq^{2k})H_{k-1}(q;a,b,c,d)
\\&+(aq^k+dq^k+adq^{2k})H_{k-2}(q;a,b,c,d)
\\&+adq^{2k-1}H_{k-3}(q;a,b,c,d),
\end{aligned}
\end{equation}
and the initial conditions
\begin{align*}
H_{-1}(q;a,b,c,d) &=1,
\\H_{-2}(q;a,b,c,d)&=0,
\\H_{-3}(q;a,b,c,d) &= \frac{(b-1)cq}{ad}.
\end{align*}
This completely determines $(H_k(q;a,b,c,d))$.

We conclude this part on Capparelli's identity with the following lemma relating $G^C_{k_d}(q;a,c,d)$ and the evaluation at $b=c$ of $H_k(q;a,b,c,d)$.

\begin{lemma}
\label{lem:capH}
For all $k\geq 0$,
$$\frac{G^C_{k_d}(q;a,c,d)}{(cq;q)_{k+1}} = H_k(q;a,c,c,d).$$
\end{lemma}
\begin{proof}
A simple calculation starting from \eqref{eq:recCap} shows that 
$\left(\frac{G^C_{k_d}(q;a,c,d)}{(cq;q)_{k+1}}\right)$
satisfies the same recurrence relation as $(H_k(q;a,c,c,d))$.
Moreover, we know from \eqref{eq:recCap} and \eqref{eq:cond_in_cap} that 
$$G^C_{-1_d} = 1, \quad \frac{G^C_{0_d}}{1-cq} = \frac{1}{1-cq}, \quad \frac{G^C_{1_d}}{(1-cq)(1-cq^2)}=\frac{1+aq+cq+dq+adq^2}{(1-cq)(1-cq^2)},$$
and computing the initial values for $H_k(q;a,b,c,d)$ using \eqref{eq:H} gives
\begin{equation}
\label{eq:cond_in_H}
\begin{aligned}
H_{-1}(q;a,b,c,d)&= 1, \quad H_{0}(q;a,b,c,d) = \frac{1}{1-bq},
\\H_{1}(q;a,b,c,d)&= \frac{1-bcq^2}{(1-cq)(1-bq)(1-bq^2)}+\frac{aq+dq+adq^2}{(1-cq)(1-bq^2)}.
\end{aligned}
\end{equation}
When setting $b=c$, these are exactly the same initial conditions as those of $\left( \frac{G^C_{k_d}}{(cq;q)_{k+1}} \right)$, which completes the proof.
\end{proof}

\bigskip
Let us turn to Primc's identity. We also want to establish a connection between $G^P_{k_d}(q;a,b,c,d)$ and $H_k(q;a,b,c,d)$.
To do so, we follow the same reasoning as in \cite{DoussePrimc}, except that we now keep track of the variable $b$. We obtain :

\begin{lemma}
For all $k \geq 1,$ we have
\label{equations'}
\begin{subequations}
\begin{align}
G^P_{k_{d}}-G^P_{k_{c}}=E^P_{k_{d}}&=dq^k \left(E^P_{k_{c}}+E^P_{k_{a}}+G^P_{(k-1)_{c}}\right), \label{eq1} \\
G^P_{k_{c}}-G^P_{k_{b}}=E^P_{k_{c}}&=cq^k \left(E^P_{k_{c}}+E^P_{k_{a}}+G^P_{(k-1)_{c}}\right), \label{eq2} \\
G^P_{k_{b}}-G^P_{k_{a}}=E^P_{k_{b}}&=bq^k \left(E^P_{k_{b}}+G^P_{(k-1)_{d}}\right), \label{eq3} \\
G^P_{k_{a}}-G^P_{(k-1)_{d}}=E^P_{k_{a}}&=aq^k \left(E^P_{(k-1)_{b}}+G^P_{(k-2)_{d}}\right),\label{eq4}
\end{align}
\end{subequations}
\end{lemma}
To obtain the correct values for $G^P_k$ and $E^P_k$ for all coloured integers $k$, we need the initial conditions
\begin{align*}
E^P_{0_{a}}&=E^P_{0_{c}}=E^P_{0_{d}}=0,
\\ E^P_{0_{b}}&=b,
\\ G^P_{0_{b}}&=G^P_{0_{c}}=G^P_{0_{d}}=1,
\\G^P_{{0}_a}&=G^P_{{-1}_d}=1-b.
\end{align*}

Doing the same chain of substitutions as in \cite{DoussePrimc} but keeping track of the variable $b$ at each step, we obtain the following recurrence for $k \geq 2$:
\begin{equation}
\label{eq:recPri}
\begin{aligned}
(1-cq^k)G^P_{k_{d}}&= \frac{1-bcq^{2k}}{1-bq^k}G^P_{(k-1)_{d}} 
\\&+ \frac{aq^k+dq^k+adq^{2k}}{1-bq^{k-1}}G^P_{(k-2)_{d}} +\frac{adq^{2k-1}}{1-bq^{k-2}}G^P_{(k-3)_{d}},
\end{aligned}
\end{equation}
with the initial conditions
\begin{align*}
G^P_{-1_d} &= 1-b,\\
G^P_{0_d} &= 1,\\
G^P_{1_d} &= \frac{bq}{1-bq} + \frac{(1+aq)(1+dq)}{1-cq}.
\end{align*}
This completely determines $(G^P_{k_{d}}).$

As we did for Capparelli's identity in Lemma \ref{lem:capH}, we can now give a relation between $G^P_{k_d}(q;a,b,c,d)$ and $H_k(q;a,b,c,d).$

\begin{lemma}
\label{lem:PrimcH}
For all $k\geq 0$,
$$\frac{G^P_{k_d}(q;a,b,c,d)}{1-bq^{k+1}} = H_k(q;a,b,c,d).$$
\end{lemma}
\begin{proof}
A simple calculation starting from \eqref{eq:recPri} shows that 
$\left( \frac{G^P_{k_d}(q;a,b,c,d)}{1-bq^{k+1}}\right)$
satisfies the same recurrence relation as $(H_k(q;a,b,c,d))$.
Moreover, we have
$$\frac{G^P_{-1_d}}{1-b} = 1, \quad \frac{G^P_{0_d}}{1-bq} = \frac{1}{1-bq}, \quad \frac{G^P_{1_d}}{1-bq^2}= \frac{bq}{(1-bq)(1-bq^2)} + \frac{(1+aq)(1+dq)}{(1-cq)(1-bq^2)}.$$
After simplifying the last expression, we see that these are exactly the same initial conditions as those of $\left(H_k(q;a,b,c,d) \right)$, which completes the proof.
\end{proof}

\bigskip
Finally, combining Lemma \ref{lem:capH} and Lemma \ref{lem:PrimcH} in which we set $b=c$, we get that for all $k\geq 0$,
\begin{align*}
\frac{G^C_{k_d}(q;a,c,d)}{(cq;q)_{k+1}} &= H_k(q;a,c,c,d)
\\&= \frac{G^P_{k_d}(q;a,c,c,d)}{1-cq^{k+1}}.
\end{align*}
Simplifying completes the proof of Theorem \ref{th:main}.

\section{Finite versions}
\label{sec:finite}
We now use the recurrence equations from the previous section to prove Theorems \ref{th:primc_fini} and \ref{th:capa_fini}, the finite versions of Primc's and Capparelli's identity, respectively. To do so, it is sufficient to prove the following theorem. One can then recover Theorem \ref{th:primc_fini} (resp. Theorem \ref{th:capa_fini}) by using Lemma \ref{lem:PrimcH} (resp. Lemma \ref{lem:capH}).

\begin{theorem}
\label{th:H fini}
Let $(H_k(q,a,b,c,d))$ be the sequence defined by the recurrence \eqref{eq:H} and initial conditions \eqref{eq:cond_in_H}.

We have, for every positive integer $k$,
\begin{equation}
\label{eq:formuleH}
H_k (q;a,b,c,d) = \sum_{j=0}^{k+1} \frac{u_{j}(a,b,c,d)q^{{k+1-j \choose 2}}}{(q;q)_{k+1-j}},
\end{equation}

where the sequence $(u_n(a,b,c,d))_{n \in \N}$ is defined in Theorem \ref{th:primc_fini}.
\end{theorem}

\begin{proof}
To prove this theorem, we use the same transformations as in \cite{DoussePrimc}, except that we now keep track of $b$ everywhere, giving rise to more complicated expressions. Also, instead of only determining the limit of $H_k$ when $k$ tends to infinity by using Appell's  lemma, here we solve the recurrence equation \eqref{eq:H} exactly to obtain the formula \eqref{eq:formuleH}.

We first define
$$f(x):= \sum_{k \geq 0} H_{k-1} x^k.$$
Let us replace $k$ by $k-1$ in \eqref{eq:H}, multiply both sides of the equality by $x^k$, and sum over $k \geq 0$. Using \eqref{eq:cond_in_H} and the fact that
$$H_{-4}:= \frac{q^3(1-b)(ac+ad+cd)}{a^2d^2},$$
we obtain
\begin{equation}
\label{eq:f}
\begin{aligned}
(1-x)f(x) &= \left(b+\frac{c}{q}+ax^2q+dx^2q \right)f(xq)-(1+xq)\left(\frac{bc}{q}-adx^2q^2\right)f(xq^2)
\\&+ (b-1) \left(cx+\frac{c}{q}-1 \right),
\end{aligned}
\end{equation}
and the initial conditions
\begin{align*}
f(0)&= H_{-1} = 1,
\\ f'(0) &= H_{0}=\frac{1}{1-bq}.
\end{align*}

We now make a change of unknown function:
$$g(x):= \frac{f(x)}{(-x;q)_{\infty}}.$$
We obtain
\begin{equation}
\label{eq:g}
\begin{aligned}
(1-x^2)g(x) &= \left(b+\frac{c}{q}+ax^2q+dx^2q\right)g(xq)-\left(\frac{bc}{q}-adx^2q^2\right)g(xq^2)
\\&+ \frac{b-1}{(-xq;q)_{\infty}} \left(cx+\frac{c}{q}-1 \right),
\end{aligned}
\end{equation}
and
\begin{align*}
g(0)&=f(0) = 1,
\\ g'(0) &= f'(0) - \frac{f(0)}{1-q}=\frac{1}{1-bq} - \frac{1}{1-q}= \frac{q(b-1)}{(1-bq)(1-q)}.
\end{align*}

We switch back to recurrence equations by defining the sequence $(u_n)$ by
$$\sum_{n \geq 0} u_n x^n := g(x).$$
After some computations, we find that $(u_n)$ satisfies the recurrence
\begin{equation}
\label{eq:u}
u_n = \frac{\left(1+aq^{n-1}\right)\left(1+dq^{n-1}\right)}{\left(1-bq^n\right)\left(1-cq^{n-1}\right)} u_{n-2} + \frac{(-1)^nq^n(1-b)}{(1-bq^n)(q;q)_n},
\end{equation}
and the initial conditions
\begin{align*}
u_0&=g(0) = 1,
\\ u_1 &= g'(0)=\frac{q(b-1)}{(1-bq)(1-q)}.
\end{align*}
Iterating, we get that for all $n \geq 0$,
$$
u_{2n}(a,b,c,d) = (1-b) \sum_{\ell=0}^n \frac{(-aq^{2\ell+1};q^2)_{n-\ell}(-dq^{2\ell+1};q^2)_{n-\ell}}{(bq^{2\ell};q^2)_{n-\ell+1}(cq^{2\ell+1};q^2)_{n-\ell}} \frac{q^{2\ell}}{(q;q)_{2\ell}},
$$
and
$$
u_{2n+1}(a,b,c,d) = (b-1) \sum_{\ell=0}^n \frac{(-aq^{2\ell+2};q^2)_{n-\ell}(-dq^{2\ell+2};q^2)_{n-\ell}}{(bq^{2\ell+1};q^2)_{n-\ell+1}(cq^{2\ell+2};q^2)_{n-\ell}} \frac{q^{2\ell+1}}{(q;q)_{2\ell+1}}.
$$

Finally, let us backtrack through all of the transformations to express $H_k$ in terms of the $u_n$'s.
We have:
\begin{align*}
\sum_{k \geq 0} H_{k-1} x^k &= f(x)\\
&= (-x;q)_{\infty} g(x)\\
&= \sum_{n \geq 0} \frac{x^n q^{{n\choose 2}}}{(q;q)_n} \times \sum_{n \geq 0} u_n x^n\\
&= \sum_{k \geq 0} \left( \sum_{j=0}^k \frac{u_j q^{{k-j\choose 2}}}{(q;q)_{k-j}} \right)  x^k,
\end{align*}
where the penultimate equality follows from Equation (19) in \cite{Gasper}.

Equating the coefficients of $x^{k+1}$ on both sides completes the proof.

\end{proof}

\section{Bijective proof of Theorem~\ref{th:comb}}
\label{sec:bij}
In this section, we prove Theorem \ref{th:comb} bijectively. Namely, we establish a one-to-one correspondence between partition pairs counted by $\mathcal{C}(n;k;i,j,\ell)$ and partitions counted by $\mathcal{P}(n;k;i,j,\ell)$.

Let $(\lambda,\mu)$ be a partition pair of total weight $n$, where $\lambda \in \mathcal{C}$ and $\mu$ is an unrestricted partition coloured $c$, having in total $i$ parts coloured $a$, $j$ parts coloured $c$, $\ell$ parts coloured $d$, and largest part at most $k$. We transform $(\lambda,\mu)$ into a partition in $\mathcal{P}$ by following the steps below.

To make the bijection easier to follow, we will illustrate each step on the example
\begin{align*}
\lambda &= 8_d+8_a+6_c+5_c+3_d+1_a,
\\ \mu &= 8_c +8_c + 7_c+ 5_c+3_c+2_c+2_c+1_c+1_c.
\end{align*}

\textbf{Step $0$:} Change the colour of all the parts of $\mu$ to $b$. We obtain a partition pair $(\lambda, \mu').$

On our example, we get
\begin{align*}
\lambda &= 8_d+8_a+6_c+5_c+3_d+1_a,
\\ \mu' &= 8_b +8_b + 7_b+ 5_b+3_b+2_b+2_b+1_b+1_b.
\end{align*}
This process is clearly reversible.

\medskip

\textbf{Step $1$:} Insert the parts of $\mu'$ in the partition $\lambda$ according to the order \eqref{orderprimc} of Primc's identity. Call $\nu_1$ the resulting partition.

In our example, we obtain
$$\nu_1= 8_d+8_b+8_b+8_a+7_b+6_c+5_c+5_b+3_d+3_b+2_b+2_b+1_b+1_b+1_a.$$
This process is also clearly reversible, as one can simply separate the $b$-parts from the rest to recover the partitions $\lambda$ and $\mu'$.

The partition $\nu_1$ satisfies the difference conditions in the matrix
\begin{equation} \label{M1matrix}
M_1=\bordermatrix{\text{} & a & b & c & d \cr a & 2&1&2&2 \cr b &0&0&1&1 \cr c &1&0&1&2 \cr d&0&0&1&2},
\end{equation}
together with the following additional conditions for all $m \geq 1$:
\begin{enumerate}[label=(\subscript{C}{{\arabic*}})]
    \item $m_a$ and $(m-1)_a$ cannot both be parts of $\nu_1$,
    \item $m_c$ and $m_a$ cannot both be parts of $\nu_1$,
    \item $m_c$ and $(m-1)_d$ cannot both be parts of $\nu_1$,
    \item $m_d$ and $(m-1)_d$ cannot both be parts of $\nu_1$.
\end{enumerate}
The matrix $M_1$ is obtained by simply using the order \eqref{orderprimc} and the difference condition matrix $C$ from \eqref{Capmatrix}. The additional conditions come from the fact that if we only considered the matrix $M_1$, we would allow too many partitions due to the presence of the parts coloured $b$.
For example, by \eqref{Capmatrix}, there is a difference of at least $2$ between two consecutive parts coloured $a$ in $\lambda$. So $\nu_1$ cannot contain both the parts $m_a$ and $(m-1)_a$. However, the matrix $M_1$ alone would allow the subpartition $m_a +(m-1)_b + \cdots + (m-1)_b + (m-1)_a$. This is why the condition $(C_1)$ is needed. By a case by case analysis of all the possible problems of this type, we obtain the other conditions $(C_2)-(C_4).$

Note that in $\nu_1$, the $c$-parts can only appear once, while the $b$-parts can repeat.

\medskip

\textbf{Step $2$:}
By the difference conditions satisfied by $\nu_1$, if $m_a$ or $m_d$ appears in $\nu_1$ (they can both appear at the same time), then $m_c$ cannot appear, but $m_b$ can appear arbitrarily many times. If there are such $m_b$'s, transform them all into $m_c$'s.
Call $\nu_2$ the resulting partition.

In our example, we obtain
$$\nu_2= 8_d+8_c+8_c+8_a+7_b+6_c+5_c+5_b+3_d+3_c+2_b+2_b+1_c+1_c+1_a.$$
This process is again reversible: to obtain $\nu_1$ from $\nu_2$, change all the $m_c$'s which appear at the same time as a $m_a$ or $m_d$ into $m_b$'s.

The partition $\nu_2$ satisfies the difference conditions in the matrix
\begin{equation} \label{M2matrix}
M_2=\bordermatrix{\text{} & a & b & c & d \cr a & 2&1&2&2 \cr b &1&0&1&1 \cr c &0&0&0&2 \cr d&0&1&0&2},
\end{equation}
together with the following additional conditions for all $m \geq 1$:
\begin{enumerate}[label=(\subscript{C'}{{\arabic*}})]
    \item $m_d$ and $m_b$ cannot both be parts of $\nu_2$,
    \item $m_c$ can repeat if and only if it appears at the same time as $m_d$ or $m_a$,
    \item $m_c$ and $(m-1)_d$ cannot both be parts of $\nu_2$.
\end{enumerate}

The fact that we transformed $b$-parts into $c$-parts after a $d$-part (resp. before an $a$-part) corresponds ``locally'' to exchanging the entries $(d,b)$ and $(d,c)$ (resp. $(b,a)$ and $(c,a)$) in $M_1$. But we should be careful, because the matrix $M_2$ alone would allow subpartitions of the form $m_d+m_c+ \cdots + m_c +m_b$, therefore we add condition $(C'_1)$ to avoid these problems.
Condition $(C'_2)$ comes from the fact that in $\nu_1$, the $c$-parts were not allowed to repeat, so that in $\nu_2$ the only $c$-parts that can repeat are those that were introduced at Step 2.
Condition $(C'_3)$ is simply condition $(C_3)$ which still holds after the process of Step $2$.
Finally, we removed conditions $(C_1)$ and $(C_4)$ because they are now contained in the conditions of the matrix $M_2$, and $(C_2)$ disappears in the process of Step 2.

\medskip

\textbf{Step $3$:}
If in $\nu_2$ there is a part $m_c$ followed by an arbitrary number of parts $m_b$, then change all these parts to $m_c$. Call $\nu_3$ the resulting partition.

In our example, we obtain
$$\nu_3= 8_d+8_c+8_c+8_a+7_b+6_c+5_c+5_c+3_d+3_c+2_b+2_b+1_c+1_c+1_a.$$

This step is also reversible. To obtain $\nu_2$ from $\nu_3$, search for all the parts $m_c$ that repeat but do not appear at the same time as $m_a$ or $m_d$, and change the colour of all but the first of these $c$-parts to $b$.

Now we claim that $\nu_3 \in \mathcal{P}$.
After the transformation from Step $3$, any part $m_c$ can repeat arbitrarily many times, but cannot appear at the same time as $m_b$. This amounts to changing the entry $(c,b)$ to $1$ in \eqref{M2matrix} and dropping condition $(C'_2)$. The difference condition matrix we obtain is exactly $P$ \eqref{Primcmatrix}.
Moreover, now that $m_b$ cannot appear after $m_c$ anymore, it is also impossible that $m_d$ and $m_b$ (resp. $m_c$ and $(m-1)_d$) appear at the same time, so we also drop condition $(C'_1)$ (resp. $(C'_3)$). Lastly, we check that we do not create any new condition which is not encoded by the matrix $P$.

\medskip
Finally, all the steps of this transformation preserve the order \eqref{orderprimc}, the weight $n$, the largest part $k$, the number $i$ of parts coloured $a$, the number $\ell$ of parts coloured $d$, and the number $j$ of parts coloured $b$ or $c$. Indeed, the only transformations we did were changing some colours from $c$ to $b$ and $b$ to $c$.
Therefore, we have established a bijection between the partitions counted by $\mathcal{C}(n;k;i,j,\ell)$ and those counted by $\mathcal{P}(n;k;i,j,\ell)$. Theorem \ref{th:comb} is proved. \qed

\begin{remark}
Combining this bijection with the bijection of Capparelli's identity due to Alladi, Andrews, and Gordon \cite{AllAndGor} yields a bijective proof of Primc's identity which does not keep track of the number of parts coloured $b$ and $c$, i.e. a bijective proof of the following theorem.
\begin{theorem}
Let $P(n;i,\ell)$ denote the number of four-coloured partitions of $n$ with the ordering \eqref{orderprimc} and matrix of difference conditions \eqref{Primcmatrix}, having $i$ parts coloured $a$ and $\ell$ parts coloured $d$.

Let $P'(n;i,\ell)$ denote the number of $4$-tuples $(\lambda,\mu,\nu,\chi)$ of partitions of weight $n$ such that $\lambda$ is a partition into distinct odd parts coloured $a$, $\mu$ is a general partition, $\nu$ is a partition into odd parts, and $\chi$ is a partition into distinct odd parts coloured $d$, having $i$ parts coloured $a$ and $\ell$ parts coloured $d$.

Then for all $n,i,\ell$,
$$P(n;i,\ell)= P'(n;i,\ell).$$
\end{theorem}

This theorem is equivalent to Theorem \ref{th:primcnondil} where we set the variables $b$ and $c$ to be equal to $1$.

By doing the transformations \eqref{Primcdilation} before performing the bijection, this gives a bijective proof of Primc's original Theorem \ref{th:primc}.
\end{remark}

\section{Discussion on crystals}
\label{sec:crystal}
In \cite{Primc}, Primc obtained several partition identities by studying the energy matrix of perfect crystals for $A_{n-1}^{(1)}$ coming from the tensor product of the vector representation and its dual.

In particular, the difference condition matrix $P$ in Theorem \ref{th:primc} is the energy matrix of the perfect $A_1^{(1)}$-crystal in Figure $1$, where the vertices $1,2,3,4$ correspond to colours $a,b,c,d$, respectively.
\begin{figure}[H]
\label{fig:1}
\includegraphics[width=0.2\textwidth]{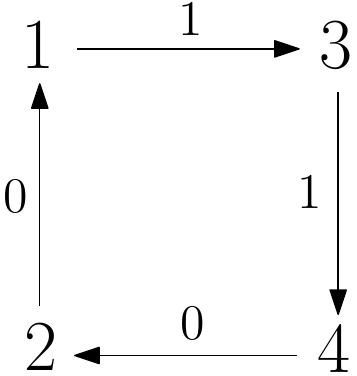}
\caption{The perfect crystal corresponding to Primc's identity}
\end{figure}

As mentioned in the introduction, the matrix $C$ from our reformulation of Capparelli's identity (Theorem \ref{th:capa-dou}) also appears in Primc's paper \cite{Primc} as the energy matrix of the ``almost perfect'' $A_1^{(1)}$--crystal in Figure $2$, where the vertices $1,2,3$ correspond to colours $a,c,d$, respectively.
\begin{figure}[H]
\label{fig:2}
\includegraphics[width=0.4\textwidth]{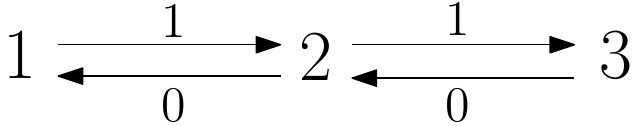}
\caption{The ``almost perfect'' crystal corresponding to Capparelli's identity}
\end{figure}
Though Primc thought that this identity was a priori not related to crystal base theory, we showed that there is actually a close connection between the two identities, as Theorem \ref{th:capa-dou} is (up to the $\frac{1}{(cq;q)_{\infty}}$ factor) the particular case $b=c$ in Theorem \ref{th:primcnondil}. Therefore Theorem \ref{th:capa-dou} is also related to crystal bases in some way.

From the crystal base point of view, we do not claim to completely explain this fact. However, we make an observation which might be a first step towards a better understanding of the situation.


If, in Figure 1, we merge the vertices $2$ and $3$ (keeping all the in- and out-edges) into a new vertex $2$ and rename the vertex $4$ as $3$, we obtain exactly Figure $2$. Given that the vertex $2$ corresponded to the colour $b$ and the vertex $3$ to the colour $c$, merging these two vertices seems to correspond to setting $b=c$ in Theorem \ref{th:primcnondil}. The factor $\frac{1}{(cq;q)_{\infty}}$ may then come from the fact that we ``lose" a vertex in the process.

This leads us to wonder if it is possible to do such merging transformations in a more general setting in the theory of crystal bases, or if this is only a particular situation. This question might be of interest to representation theorists as well.

Moreover, it would also be interesting to know how the character formula for crystal bases from \cite{KMN2} transforms under this merging transformation, and where exactly the factor $\frac{1}{(cq;q)_{\infty}}$ comes from.

So far there are -- to the best of our knowledge -- no ``coloured'' character formulas, i.e. no character formulas keeping track of the colour variables $a,b,c,\dots$. In addition to the fact that both Capparelli's and Primc's identities have weighted words versions, the present work seems to be another hint that such formulas may exist, and if they did, would be very powerful combinatorially.

\section*{Acknowledgements}
The author thanks Frédéric Chyzak, Jeremy Lovejoy, and Mirko Primc for their comments on an earlier version of this paper. She also thanks the anonymous referee for a careful reading of this paper and helpful suggestions.

\bibliographystyle{alpha}
\bibliography{biblio}

\end{document}